\documentclass{siamart251216}

\usepackage{amssymb,mathtools}
\usepackage{microtype}
\usepackage{enumitem}

\newsiamremark{remark}{Remark}

\newcommand{\Lp}{L_p}

\newcommand{\norm}[2]{\left\|#1\right\|_{#2}}

\newcommand{\R}{\mathbb{R}}

\headers{Accelerating Langevin Dynamics}{Z. Cai, W. Gan, B. Miao and Z. Shi}

\title{Accelerating Langevin Dynamics with Weighted Laplacian Flow: Global Well-Posedness and Exponential Convergence \thanks{This work was supported by National Natural Science Foundation of China (NSFC) 92370125.}}

\author{
Ziyue Cai\thanks{Qiuzhen College, Tsinghua University, Beijing, China, 100084. \email{cai-zy22@mails.tsinghua.edu.cn}}
\and
Weiye Gan\thanks{Department of Mathematical Sciences, Tsinghua University, Beijing, China, 100084. \email{gwy22@mails.tsinghua.edu.cn}}
\and
Borui Miao\thanks{Yau Mathematical Sciences Center, Tsinghua University,
Beijing, China, 100084. \email{ mbr20@mails.tsinghua.edu.cn}}
\and
Zuoqiang Shi\thanks{Yau Mathematical Sciences Center, Tsinghua University,
Beijing, China, 100084. \&
Yanqi Lake Beijing Institute of Mathematical Sciences and Applications,
Beijing, China, 101408. \email{zqshi@tsinghua.edu.cn.} Corresponding author.}}

\begin{document}
\maketitle

\begin{abstract}
We study an accelerated Langevin equation in which the classical overdamped
Langevin drift is supplemented by a nonlocal gradient velocity determined through a
target-weighted Poisson equation.  The target \(p\) is only assumed to be strictly
positive and need not be normalized.  Writing the dynamics in terms of
\(\omega=\log(p/q)\) yields a coupled parabolic--elliptic system whose parabolic
component contains the usual Langevin diffusion, whereas the elliptic component
generates a global transport.  The system is invariant under multiplication of
\(p\) by a positive constant.
On a bounded smooth domain with homogeneous Neumann boundary conditions,
we establish global existence and uniqueness of classical solutions for strictly
positive \(C^{2+\alpha}\) target weights and compatible \(C^{2+\alpha}\) initial
data.  For long-time behavior, the oscillation of \(\omega\) decays at the explicit
rate \(e^{-t}\) and the solution density exponentially converges to target density.  Moreover, an exact
entropy-dissipation identity holds relative to the target density, in
which the relative Fisher information and the Jeffreys divergence give the two
dissipation mechanisms associated with the Langevin and weighted-Poisson
components.  In particular,
KL divergence also converges to zero with rate $e^{-t}$
without a convexity assumption on the target.
\end{abstract}

\begin{keywords}
Langevin dynamics, weighted Laplacian flow, parabolic--elliptic systems,
global well-posedness, entropy dissipation, exponential convergence
\end{keywords}

\begin{AMS}
35Q84, 35B40, 35A01, 60J60
\end{AMS}


\section{Introduction}

Sampling from an unnormalized target is a fundamental problem in
Bayesian computation, molecular dynamics, statistical physics, and machine
learning.  A canonical example is the Gibbs weight
\[
p_\kappa(x)=e^{-U(x)/\kappa},
\qquad
\pi_\kappa(x)=Z_\kappa^{-1}p_\kappa(x),
\qquad
Z_\kappa=\int p_\kappa(x)\,dx,
\]
where \(U\) is an energy landscape and \(\kappa>0\) denotes the temperature.
The normalizing constant \(Z_\kappa\) is often unavailable in sampling problems.
Langevin dynamics is nevertheless well defined because its drift depends only on
the logarithmic gradient of the target, which is unchanged by multiplication by a
positive constant.  The resulting connection between stochastic sampling and
Fokker--Planck equations makes available a broad collection of analytic tools,
including functional inequalities, Wasserstein gradient-flow methods, coupling
arguments, spectral theory, and hypocoercivity.

Let \(p=e^{-U}\) denote an unnormalized target weight and set
\[
Z_p:=\int_\Omega p(x)\,dx,
\qquad
\pi:=\frac{p}{Z_p}.
\]
Since \(\nabla\log p=\nabla\log\pi\), the overdamped Langevin diffusion can be
written as
\[
dX_t=-\nabla U(X_t)\,dt+\sqrt{2}\,dW_t
=\nabla\log p(X_t)\,dt+\sqrt{2}\,dW_t.
\]
If \(q_t\) denotes the density of \(X_t\), then
\[
\partial_tq_t
=
\Delta q_t-\nabla\cdot(q_t\nabla\log p)
=
\nabla\cdot\left(
q_t\nabla\log\frac{q_t}{\pi}
\right)
=
\nabla\cdot\left(
q_t\nabla\log\frac{q_t}{p}
\right).
\]
On a bounded domain the corresponding reflected dynamics is represented at the
density level by the same Fokker--Planck equation together with a no-flux boundary
condition.  The variational formulation of Jordan, Kinderlehrer, and Otto
\cite{JKO98} identifies this evolution as the \(2\)-Wasserstein gradient flow of
the relative entropy with respect to the normalized target \(\pi\),
\[
\operatorname{KL}(q\Vert \pi)
=
\int q\log\frac{q}{\pi}\,dx.
\]
At the continuous level, the entropy satisfies
\[
\frac{d}{dt}\operatorname{KL}(q_t\Vert \pi)
=
-\int q_t
\left|\nabla\log\frac{q_t}{\pi}\right|^2dx
=
-\int q_t
\left|\nabla\log\frac{q_t}{p}\right|^2dx.
\]
Consequently, a logarithmic Sobolev inequality yields exponential relaxation in
relative entropy.  This mechanism is classical for strongly log-concave targets
through the Bakry--\'Emery criterion \cite{BE85}; Poincar\'e and transportation
inequalities provide related \(L^2\) and Wasserstein estimates
\cite{OV00}.  Quantitative convergence beyond global convexity can be obtained by
Lyapunov, coupling, and spectral methods; see, for example,
\cite{Eberle16,RT96}.

The discretization of Langevin diffusions has also been studied extensively. 
Non-asymptotic estimates for the unadjusted Langevin algorithm and related schemes
have been established in total variation, Wasserstein distance, relative entropy,
and R\'enyi divergence under a variety of convexity and regularity assumptions;
representative contributions include
\cite{Dal17,DM17,DM19,VW19,RRT17,CCAY18,MFWB22,CDJB20}.
These results demonstrate the robustness of the Langevin framework, while also
showing that its quantitative performance may deteriorate substantially when the
target is strongly nonconvex.

A principal obstruction is metastability.  For multimodal Gibbs measures with
well-separated local minima, transitions between metastable regions may occur on
exponentially long time scales.  Eyring--Kramers and potential-theoretic analyses
make this dependence precise \cite{BEGK04,MS14}.  In the
small-temperature regime \(p_\kappa\propto e^{-U/\kappa}\), characteristic
transition times may scale exponentially in \(1/\kappa\).  This phenomenon motivates
the construction of modified dynamics in which the target distribution is retained
while the transport between distant regions of state space is enhanced.

Several acceleration mechanisms have been proposed.  Underdamped and kinetic
Langevin methods introduce momentum and exploit hypocoercive interactions between
transport and dissipation \cite{CCBJ18,DRD20,MaEtAl21,MouEtAl21}.  Nonreversible
perturbations retain the invariant measure while adding a divergence-free
circulation. The effects on spectral gaps, convergence, and asymptotic variance
have been studied in \cite{HHMS93,LNP13,RBS15,DLP16,LS18}.  At the density level,
birth--death dynamics provides a genuinely nonlocal acceleration mechanism
\cite{LLN19,LSW23}.  A related principle arises in enhanced dissipation:
conservative transport can redistribute a scalar so that diffusion acts more
effectively, and quantitative forms of this mechanism have been developed in
\cite{CKRZ08,FI19,CDE20,FHP10,Wei21}.

The present work is motivated by Weighted Laplacian Flow (WLF)
\cite{GWS26}, which constructs a deterministic transport velocity through a
 target-weighted elliptic problem.  
Consider Langevin dynamics with an auxiliary velocity field
\[
dX_t=
\bigl(\nabla\log p(X_t)+v(t,X_t)\bigr)\,dt+\sqrt{2}\,dW_t.
\]
The associated Fokker--Planck equation implies that
\[
\partial_t\omega
+
v\cdot\nabla\omega
+
|\nabla\omega|^2
=
L_p\omega+\frac1p\nabla\cdot(pv).
\]
Based on the weighted Laplacian flow, we choose \(v=\nabla\phi\), where \(\phi\) solves
\[
L_p\phi=-(\omega-\bar\omega).
\]
Here \(\bar\omega\) is the unique scalar for which the Neumann problem is
compatible, namely,
\[
\int_\Omega p(\omega-\bar\omega)\,dx=0.
\]
Thus \(\bar\omega\) may be regarded as an additional scalar unknown determined
simultaneously with \(\phi\).
The resulting parabolic--elliptic system is
\[
\left\{
\begin{aligned}
L_p\phi&=-(\omega-\bar\omega),\\
\partial_t\omega+\nabla\phi\cdot\nabla\omega+|\nabla\omega|^2
&=L_p\omega-(\omega-\bar\omega).
\end{aligned}
\right.
\]
Homogeneous Neumann conditions are imposed on both \(\phi\) and \(\omega\), and the
additive constant in \(\phi\) is fixed by a zero-mean normalization.  In this
formulation, \(L_p\omega\) represents the local Langevin relaxation, whereas
the weighted-Poisson equation produces a nonlocal transport correction.

The model possesses two complementary coercive structures.  The first is a decay
for \(\omega\), which yields exponential contraction of its spatial oscillation.
The second is an exact entropy-dissipation identity.  For a normalized density
\(q\), define
\[
I(q\Vert \pi)
:=
\int_\Omega q\left|\nabla\log\frac{q}{\pi}\right|^2dx
=
\int_\Omega q\left|\nabla\log\frac{q}{p}\right|^2dx
\]
and
\[
J(q,\pi)
:=
\operatorname{KL}(q\Vert \pi)+\operatorname{KL}(\pi\Vert q).
\]
We prove in Section~\ref{sec:long-time} that
\begin{equation}
\label{eq:intro-hybrid-entropy}
\frac{d}{dt}\operatorname{KL}(q_t\Vert \pi)
=
-I(q_t\Vert \pi)-J(q_t,\pi).
\end{equation}
The two terms on the right-hand side correspond, respectively, to the Langevin
diffusion and the weighted-Poisson transport.  Since
\(J(q,\pi)\ge \operatorname{KL}(q\Vert \pi)\), identity
\eqref{eq:intro-hybrid-entropy} yields exponential convergence in relative entropy
with exactly the same rate as in the normalized formulation.

The principal contributions of this paper are as follows.
\begin{enumerate}[label=\textup{(\roman*)},leftmargin=2.5em]
\item We derive the accelerated Langevin equation from an overdamped Langevin
      diffusion with an auxiliary gradient velocity determined by a weighted
      Neumann problem.
\item We establish global classical well-posedness by a contraction
      argument and prove positive-time classical regularization.
\item We prove exponential contraction of
      \(\operatorname{osc}_\Omega\omega(t)\) and establish the exact
      Jeffreys--Fisher dissipation law.
\end{enumerate}

The remainder of the paper is organized as follows.  Section~\ref{sec:model}
formulates the model and states the main well-posedness theorem.  Sections~3--5
develop the local Sobolev theory, global a priori estimates, and the classical
regularity upgrade.  Section~\ref{sec:long-time} establishes the long-time
convergence and entropy-dissipation results.  Appendices~\ref{app:semigroup}--\ref{app:dini}
collect the heat-semigroup estimates, fractional Gr\"onwall inequalities, and
Dini-derivative arguments used in the proofs.

\section{Model formulation and main result}
\label{sec:model}

This section formulates the accelerated Langevin equation considered throughout
the paper.  The derivation starts from an overdamped Langevin diffusion with an
auxiliary velocity field and subsequently specifies this field through a
target-weighted Poisson equation.

\subsection{Notation and standing assumptions}

Let \(\Omega\subset\R^d\) be a bounded \(C^{2+\alpha}\) domain for some
\(0<\alpha<1\).  The target is represented by a strictly positive weight
\(p\), not necessarily normalized, satisfying
\begin{equation}
\label{eq:p-assumptions}
p\in C^{2+\alpha}(\overline\Omega),\qquad
0<p_*\le p(x)\le p^*.
\end{equation}

For a positive density \(q=q(t,x)\), define
\begin{equation}
\label{eq:omega-def}
\omega(t,x):=\log\frac{p(x)}{q(t,x)},
\qquad
q(t,x)=p(x)e^{-\omega(t,x)}.
\end{equation}
We write
\begin{equation}
\label{eq:Lp-def-model}
\Lp h
:=
\frac1p\nabla\cdot(p\nabla h)
=
\Delta h+\nabla\log p\cdot\nabla h
\end{equation}
for the weighted Laplacian associated with \(p\).

For a function \(w\), we denote by \(\bar w\) the unique scalar determined by
the compatibility condition
\begin{equation}
\label{eq:p-mean-model}
\int_\Omega p(x)\bigl(w(x)-\bar w\bigr)\,dx=0.
\end{equation}
Equivalently, \(\bar w\) is the constant that makes the weighted Neumann problem
\[
\Lp\psi=-(w-\bar w),
\qquad
\partial_n\psi=0,
\]

\subsection{Augmented Langevin dynamics}

Consider an overdamped Langevin diffusion supplemented by a time-dependent velocity
field \(v=v(t,x)\).  The drift depends on the target only through
\(\nabla\log p\), which is invariant under multiplication of \(p\) by a positive
constant:
\begin{equation}
\label{eq:augmented-langevin-sde}
dX_t
=
\bigl(\nabla\log p(X_t)+v(t,X_t)\bigr)\,dt
+
\sqrt{2}\,dW_t.
\end{equation}
On the bounded domain \(\Omega\), the reflecting boundary condition will be imposed
at the level of the associated Fokker--Planck equation.  Its interior evolution is
\begin{equation}
\label{eq:augmented-fokker-planck}
\partial_t q
+
\nabla\cdot\bigl(q(\nabla\log p+v)\bigr)
=
\Delta q.
\end{equation}

Consequently, evolution of $\omega$ satisfies
\begin{equation}
\label{eq:omega-general-v}
\partial_t\omega
+
v\cdot\nabla\omega
+
|\nabla\omega|^2
=
\Lp\omega
+
\frac1p\nabla\cdot(pv).
\end{equation}
Thus the auxiliary velocity contributes a transport term together with its
\(p\)-weighted divergence.

We restrict the auxiliary velocity to gradient form,
\begin{equation}
\label{eq:v-gradient-choice}
v=\nabla\phi,
\end{equation}
and determine the potential \(\phi\) through
\begin{equation}
\label{eq:wlf-linear-model}
\Lp\phi
=
-(\omega-\bar\omega).
\end{equation}
The scalar \(\bar\omega\) is determined together with \(\phi\) by the
compatibility condition \eqref{eq:p-mean-model}.  Then
\eqref{eq:omega-general-v} reduces to
\begin{equation}
\label{eq:hybrid-omega-linear-model}
\partial_t\omega
+
\nabla\phi\cdot\nabla\omega
+
|\nabla\omega|^2
=
\Lp\omega-(\omega-\bar\omega).
\end{equation}
In addition to the diffusion term $\Lp \omega$, the auxiliary velocity introduces an explicit dissipation term which will imply exponential convergence.

We impose homogeneous Neumann boundary conditions
\begin{equation}
\label{eq:model-neumann}
\partial_n\phi=\partial_n\omega=0
\qquad\text{on }\partial\Omega.
\end{equation}
The PDE model we proposed is therefore
\begin{equation}
\label{eq:system}
\left\{
\begin{aligned}
\Lp\phi&=-(\omega-\bar\omega),
&&\text{in }\Omega,\\
\partial_t\omega+\nabla\phi\cdot\nabla\omega+|\nabla\omega|^2
&=\Lp\omega-(\omega-\bar\omega),
&&\text{in }\Omega,\\
\partial_n\phi=\partial_n\omega&=0,
&&\text{on }\partial\Omega,\\
\int_\Omega\phi\,dx&=0,\\
\omega(0,\cdot)&=\omega_0.
\end{aligned}
\right.
\end{equation}
The elliptic equation determines the auxiliary velocity instantaneously from the
current logarithmic density-ratio discrepancy, while the parabolic equation governs
the coupled local and nonlocal relaxation. Indeed,
the PDE system \eqref{eq:system} is not a separate sampling dynamics. Rather, it
provides a way to construct, from the current logarithmic density ratio
\(\omega=\log(p/q)\), the nonlocal velocity field.
With the nonlocal velocity field, the accelerated Langevin dynamics is
\begin{equation}
dX_t
=
\bigl(\nabla\log p(X_t)+\nabla\phi_t(X_t)\bigr)\,dt
+
\sqrt{2}\,dW_t,\nonumber
\end{equation}
together with reflection at \(\partial\Omega\).  Its Fokker--Planck equation is exactly the density
equation obtained from \eqref{eq:system} under the transformation
\[
q_t=p e^{-\omega(t,\cdot)}.
\]
Consequently, if \(X_0\) has density \(q_0\), then \(X_t\) has density \(q_t\). The evolution of $q_t$ can be recovered exactly by the evolution of $\omega$. So we will focus on the property of the PDE system \eqref{eq:system}. 

One of the main results is the global well-posedness.
\begin{theorem}[Global classical well-posedness]
\label{thm:main}
Assume \eqref{eq:p-assumptions} and let
\[
\omega_0\in C^{2+\alpha}(\overline\Omega),
\qquad
\partial_n\omega_0=0
\quad\text{on }\partial\Omega.
\]
Then the system \eqref{eq:system} admits a unique global classical solution
\((\omega,\phi)\).  More precisely, for every \(T>0\) and every
\(0<\beta<\alpha\),
\[
\omega\in
C^{1+\beta/2,\,2+\beta}([0,T]\times\overline\Omega),
\qquad
\phi\in
C^{\beta/2}\!\left([0,T];C^{2+\beta}(\overline\Omega)\right).
\]
\end{theorem}
To prove the global well-posedness, we first establish local well-posedness in weak sense in Section 3. Then global well-posedness in weak sense is proved in Section 4. Finally, the weak solution is lifted to strong solution in Section 5. 

\section{Local well-posedness in $W^{1,q}$}

Fix \(q>\max\{d,2\}\).  The local theory is developed after eliminating the
elliptic variable through the solution operator associated with the weighted
Neumann problem.  The resulting scalar parabolic equation is treated by a
contraction argument in \(C([0,T];W^{1,q}(\Omega))\).

\subsection{The elliptic solution operator}

For \(w\in L^q(\Omega)\), let \(\Phi[w]\) be the solution of
\begin{equation}
\label{eq:elliptic-Phi}
\left\{
\begin{aligned}
\Lp\Phi[w]&=-(w-\bar w),\\
\partial_n\Phi[w]&=0,\\
\int_\Omega\Phi[w]\,dx&=0.
\end{aligned}
\right.
\end{equation}
The compatibility condition is automatic.

\begin{lemma}[Elliptic estimates]
\label{lem:elliptic}
Let \(1<q<\infty\).  There exists a constant
\(C=C(\Omega,p,q)>0\) such that, for every \(w,w_1,w_2\in L^q(\Omega)\),
\begin{align}
\norm{\Phi[w]}{W^{2,q}}
&\le C\norm{w}{L^q},
\label{eq:ellW2q}\\
\norm{\Phi[w_1]-\Phi[w_2]}{W^{2,q}}
&\le C\norm{w_1-w_2}{L^q}.
\label{eq:ell-difference}
\end{align}
If \(q>d\), then, possibly after enlarging \(C\),
\begin{align}
\norm{\nabla\Phi[w]}{L^\infty}
&\le C\norm{w}{L^q},
\label{eq:ell-Linfty}\\
\norm{\nabla(\Phi[w_1]-\Phi[w_2])}{L^\infty}
&\le C\norm{w_1-w_2}{L^q}.
\label{eq:ell-Linfty-diff}
\end{align}
\end{lemma}

\begin{proof}
The operator \(\Lp\) has the divergence-form representation
\[
\Lp\Phi=\frac1p\nabla\cdot(p\nabla\Phi).
\]
By \eqref{eq:p-assumptions}, the associated Neumann problem is uniformly elliptic
with \(C^{1+\alpha}\) coefficients.  The compatibility condition has already been
verified, and the normalization \(\int_\Omega\Phi\,dx=0\) removes the kernel of
constants.  Standard \(W^{2,q}\)-regularity for the Neumann problem therefore gives
\[
\norm{\Phi[w]}{W^{2,q}}
\le C\norm{w-\bar w}{L^q}.
\]
By the compatibility condition \eqref{eq:p-mean-model},
\[
|\bar w|\int_\Omega p\,dx
=
\left|\int_\Omega pw\,dx\right|
\le
\norm{p}{L^{q'}}\norm{w}{L^q}.
\]
Since \(p\ge p_*>0\),
\[
\int_\Omega p\,dx\ge p_*|\Omega|,
\]
and therefore
\[
|\bar w|\le C\norm{w}{L^q},
\qquad
\norm{w-\bar w}{L^q}\le C\norm{w}{L^q}.
\]
This proves \eqref{eq:ellW2q}.  Since \(\Phi\) depends linearly on its datum,
the same estimate applied to
\(\Phi[w_1]-\Phi[w_2]=\Phi[w_1-w_2]\) yields
\eqref{eq:ell-difference}.  If \(q>d\), the Sobolev embedding
\[
W^{2,q}(\Omega)\hookrightarrow W^{1,\infty}(\Omega)
\]
implies \eqref{eq:ell-Linfty} and \eqref{eq:ell-Linfty-diff}.
\end{proof}

\subsection{Local solvability by a contraction argument}

Let \(S(t)=e^{tL_p}\) denote the Neumann heat semigroup generated by \(L_p\). Eliminating
\(\phi\) by \(\phi=\Phi[\omega]\), the scalar equation becomes
\begin{equation}
\label{eq:omega-reduced}
\partial_t\omega
=
L_p\omega
-\nabla\Phi[\omega]\cdot\nabla\omega
-|\nabla\omega|^2
-(\omega-\bar\omega).
\end{equation}
Hence a mild solution satisfies
\begin{align}
\omega(t)
={}&S(t)\omega_0
-\int_0^t S(t-s)
\bigl(\nabla\Phi[\omega(s)]\cdot\nabla\omega(s)\bigr)\,ds
\nonumber\\
&-\int_0^tS(t-s)|\nabla\omega(s)|^2\,ds
-\int_0^tS(t-s)(\omega(s)-\bar\omega(s))\,ds.
\label{eq:mild}
\end{align}

We use the following semigroup estimates, proved in detail in
Appendix~\ref{app:semigroup}; see in particular
\eqref{eq:appendix-seven}--\eqref{eq:appendix-eight}. For \(0<t\le1\),
\begin{align}
\norm{S(t)g}{W^{1,q}}
&\le Ct^{-1/2}\norm{g}{L^q},
\label{eq:semi-Lq}\\
\norm{S(t)g}{W^{1,q}}
&\le Ct^{-\theta}\norm{g}{L^{q/2}},
\qquad
\theta:=\frac12+\frac{d}{2q}<1.
\label{eq:semi-Lqh}
\end{align}
In addition \(S(t)\) is strongly continuous and locally bounded on \(W^{1,q}\):
\begin{equation}
\label{eq:semi-W1qbounded}
\sup_{0\le t\le1}
\norm{S(t)g}{W^{1,q}}
\le C\norm{g}{W^{1,q}}.
\end{equation}

Set
\[
X_T:=C([0,T];W^{1,q}(\Omega)),
\qquad
\norm{w}{X_T}:=\sup_{0\le t\le T}\norm{w(t)}{W^{1,q}},
\]
and, for \(R>0\),
\[
B_R:=\{w\in X_T:\norm{w}{X_T}\le R\}.
\]
Define \(\Gamma:X_T\to X_T\) by the right-hand side of \eqref{eq:mild}.

\subsubsection{The self-map estimate}

Let \(w\in B_R\). By Lemma~\ref{lem:elliptic},
\[
\norm{\nabla\Phi[w]}{L^\infty}\le C\norm{w}{L^q}\le CR.
\]
Thus
\begin{equation}
\label{eq:transport-bound}
\norm{\nabla\Phi[w]\cdot\nabla w}{L^q}
\le CR^2.
\end{equation}
Moreover,
\begin{equation}
\label{eq:square-bound}
\norm{|\nabla w|^2}{L^{q/2}}
=
\norm{\nabla w}{L^q}^2
\le R^2,
\end{equation}
and
\begin{equation}
\label{eq:linear-bound}
\norm{w-\bar w}{L^q}\le CR.
\end{equation}

Using \eqref{eq:semi-W1qbounded}, \eqref{eq:semi-Lq},
\eqref{eq:semi-Lqh}, and \(T\le1\), we obtain
\begin{align}
\label{eq:selfmap}
\norm{\Gamma(w)}{X_T}
&\le
C\norm{\omega_0}{W^{1,q}}
+
CR^2\int_0^T(T-s)^{-1/2}\,ds
\nonumber\\
&\quad
+
CR^2\int_0^T(T-s)^{-\theta}\,ds
+
CR\int_0^T(T-s)^{-1/2}\,ds\nonumber\\
&\le
C\norm{\omega_0}{W^{1,q}}
+
C T^{1/2}(R+R^2)
+
C T^\delta R^2.
\end{align}
with 
\begin{equation}
\label{eq:delta}
\delta:=1-\theta
=\frac12-\frac{d}{2q}
=\frac{q-d}{2q}>0.
\end{equation}
The corresponding time-integrability calculation is recorded in
Appendix~\ref{app:semigroup}, cf. \eqref{eq:time-integral-theta}.
Set
\[
R=2C\norm{\omega_0}{W^{1,q}}+1.
\]
Then \(T>0\) can be chosen sufficiently small, depending only on \(R\) and the fixed
coefficients, so that \(\Gamma(B_R)\subset B_R\).

\subsubsection{The contraction estimate}

Let \(w_1,w_2\in B_R\) and \(z=w_1-w_2\). First,
\begin{align}
|\nabla w_1|^2-|\nabla w_2|^2
&=(\nabla w_1+\nabla w_2)\cdot\nabla z,
\end{align}
hence
\begin{equation}
\label{eq:square-diff}
\norm{|\nabla w_1|^2-|\nabla w_2|^2}{L^{q/2}}
\le 2R\norm{z}{W^{1,q}}.
\end{equation}

Next,
\begin{align}
&\nabla\Phi[w_1]\cdot\nabla w_1
-\nabla\Phi[w_2]\cdot\nabla w_2
\nonumber\\
&\qquad
=
\nabla(\Phi[w_1]-\Phi[w_2])\cdot\nabla w_1
+
\nabla\Phi[w_2]\cdot\nabla z.
\end{align}
Using Lemma~\ref{lem:elliptic},
\begin{equation}
\label{eq:transport-diff}
\norm{
\nabla\Phi[w_1]\cdot\nabla w_1
-\nabla\Phi[w_2]\cdot\nabla w_2
}{L^q}
\le CR\norm{z}{W^{1,q}}.
\end{equation}
Finally,
\begin{equation}
\label{eq:linear-diff}
\norm{
(w_1-\bar w_1)-(w_2-\bar w_2)
}{L^q}
\le C\norm{z}{W^{1,q}}.
\end{equation}

Therefore
\begin{equation}
\label{eq:contraction}
\norm{\Gamma(w_1)-\Gamma(w_2)}{X_T}
\le
C\Bigl[(1+R)T^{1/2}+RT^\delta\Bigr]
\norm{w_1-w_2}{X_T}.
\end{equation}
Shrinking \(T\) if necessary, we may assume
\[
C\Bigl[(1+R)T^{1/2}+RT^\delta\Bigr]<1.
\]
Hence \(\Gamma\) is a strict contraction on \(B_R\).

\begin{proposition}[Local well-posedness]
\label{prop:local}
Let \(q>\max\{d,2\}\) and \(\omega_0\in W^{1,q}(\Omega)\).  Then there exists
\[
T=T\!\left(\norm{\omega_0}{W^{1,q}};\Omega,p,q\right)>0
\]
such that \eqref{eq:system} admits a unique mild solution satisfying
\[
\omega\in C([0,T];W^{1,q}(\Omega)),
\qquad
\phi=\Phi[\omega]\in C([0,T];W^{2,q}(\Omega)).
\]
Moreover, on every common local existence interval, the solution depends
continuously on the initial datum in the \(W^{1,q}\)-topology.
\end{proposition}

\begin{proof}
The estimates \eqref{eq:selfmap} and \eqref{eq:contraction} show that, after
choosing \(R\) as above and taking \(T>0\) sufficiently small, the map
\(\Gamma\) maps \(B_R\) into itself and is a strict contraction.  Banach's
fixed-point theorem therefore yields a unique fixed point
\(\omega\in B_R\), which is the mild solution of \eqref{eq:omega-reduced}.

The kernels \((t-s)^{-1/2}\) and \((t-s)^{-\theta}\) are integrable because
\(\theta<1\).  Hence the Duhamel terms are continuous at \(t=0\), and the strong
continuity of \(S(t)\) on \(W^{1,q}\) gives
\(\omega\in C([0,T];W^{1,q})\).  Lemma~\ref{lem:elliptic} then implies
\[
\phi=\Phi[\omega]\in C([0,T];W^{2,q}).
\]
Finally, the contraction estimate applied on a common ball containing the
solutions associated with two nearby initial data yields continuous dependence in
\(W^{1,q}\).
\end{proof}

\begin{remark}[Choice of the Sobolev exponent for smooth initial data]
\label{rem:large-q}
Proposition~\ref{prop:local} requires only \(q>\max\{d,2\}\).  In the proof of
Theorem~\ref{thm:main}, however, the initial datum satisfies
\(\omega_0\in C^{2+\alpha}(\overline\Omega)\), and hence belongs to
\(W^{1,q}(\Omega)\) for every finite \(q\).  We may therefore fix
\[
q>2(d+2).
\]
This stronger choice is used only in the positive-time regularity argument below.
it is not needed for the contraction argument itself.
\end{remark}

\begin{lemma}[Positive-time classical regularity]
\label{lem:positive-time-classical}
Fix \(q\) as in Remark~\ref{rem:large-q}, and let
\[
\omega\in C([0,T];W^{1,q}(\Omega)),
\qquad
\phi=\Phi[\omega],
\]
be the mild solution obtained in Proposition~\ref{prop:local}.  Then, for every
\(0<\tau<T\), there exists \(0<\beta<\alpha\) such that
\[
\omega\in
C^{1+\beta/2,\,2+\beta}
([\tau,T]\times\overline\Omega),
\qquad
\phi\in
C^{\beta/2}\!\left([\tau,T];C^{2+\beta}(\overline\Omega)\right).
\]
In particular, \((\omega,\phi)\) satisfies \eqref{eq:system} pointwise on every
compact subinterval of \((0,T]\).
\end{lemma}

\begin{proof}
Set \(r=q/2\), by Remark~\ref{rem:large-q}, \(r>d+2\).  Write
\[
\partial_t\omega=L_p\omega+\mathcal N(\omega),
\]
where
\[
\mathcal N(\omega)
:=
-\nabla\Phi[\omega]\cdot\nabla\omega
-|\nabla\omega|^2
-(\omega-\bar\omega).
\]
The regularity
\[
\omega\in C([0,T];W^{1,q}(\Omega))
\]
and Lemma~\ref{lem:elliptic} imply
\[
\Phi[\omega]\in C([0,T];W^{2,q}(\Omega)),
\qquad
\nabla\Phi[\omega]\in C([0,T];L^\infty(\Omega)).
\]
Consequently,
\[
\nabla\Phi[\omega]\cdot\nabla\omega\in C([0,T];L^q(\Omega)),
\]
\[
|\nabla\omega|^2\in C([0,T];L^{q/2}(\Omega))
=C([0,T];L^r(\Omega)),
\]
and
\[
\omega-\bar\omega\in C([0,T];L^q(\Omega)).
\]
Since \(\Omega\) is bounded and \(r=q/2<q\),
\begin{equation}
\mathcal N(\omega)\in C([0,T];L^r(\Omega)).
\label{eq:positive-time-nonlinearity}
\end{equation}

Fix \(0<\tau<T\), and choose \(\chi\in C^\infty([0,T])\) such that
\[
\chi=0\quad\text{on }[0,\tau/3],
\qquad
\chi=1\quad\text{on }[2\tau/3,T].
\]
For \(v:=\chi\omega\), the variation-of-constants formula gives
\begin{equation}
\left\{
\begin{aligned}
v_t-L_pv
&=\chi\,\mathcal N(\omega)+\chi'\omega,\\
\partial_n v&=0,\\
v(0)&=0.
\end{aligned}
\right.
\label{eq:positive-time-cutoff-problem}
\end{equation}
By \eqref{eq:positive-time-nonlinearity} and the boundedness of \(\omega\) in
\(C([0,T];W^{1,q})\), the right-hand side belongs to
\(L^r(0,T;L^r(\Omega))\).  Maximal \(L^r\)-regularity for the Neumann realization
of \(L_p\) therefore yields
\begin{equation}
v\in
W^{1,r}(0,T;L^r(\Omega))
\cap
L^r(0,T;W^{2,r}(\Omega)).
\label{eq:positive-time-maximal-regularity}
\end{equation}
Since \(v=\omega\) on \([2\tau/3,T]\), the same regularity holds for \(\omega\)
on \([\tau,T]\), after restriction from a slightly larger positive-time interval.

Because \(r>d+2\), parabolic Sobolev--Morrey embedding gives
\begin{equation}
\omega\in
C^{(1+\gamma)/2,\,1+\gamma}
([\tau,T]\times\overline\Omega),
\qquad
\gamma:=1-\frac{d+2}{r}>0.
\label{eq:positive-time-morrey}
\end{equation}
Choose \(0<\beta<\min\{\alpha,\gamma\}\).  After decreasing \(\beta\), if
necessary,
\begin{equation}
\omega,\ \nabla\omega
\in C^{\beta/2,\,\beta}
([\tau,T]\times\overline\Omega).
\label{eq:positive-time-holder}
\end{equation}

For each \(t\), elliptic Schauder estimates for
\[
L_p\phi(t)=-(\omega(t)-\bar\omega(t)),
\qquad
\partial_n\phi(t)=0,
\qquad
\int_\Omega\phi(t)\,dx=0
\]
give
\[
\norm{\phi(t)}{C^{2+\beta}}
\le C\norm{\omega(t)-\bar\omega(t)}{C^\beta}.
\]
Applying the same estimate to \(\phi(t)-\phi(s)\) and using
\eqref{eq:positive-time-morrey}--\eqref{eq:positive-time-holder}, possibly after
decreasing \(\beta\) once more, yields
\begin{equation}
\phi\in
C^{\beta/2}\!\left([\tau,T];C^{2+\beta}(\overline\Omega)\right).
\label{eq:positive-time-phi-holder}
\end{equation}
Consequently,
\begin{equation}
\mathcal N(\omega)\in
C^{\beta/2,\,\beta}([\tau,T]\times\overline\Omega).
\label{eq:positive-time-nonlinearity-holder}
\end{equation}

Applying the boundary parabolic Schauder estimate to
\[
\omega_t-L_p\omega=\mathcal N(\omega),
\qquad
\partial_n\omega=0,
\]
on a positive-time cylinder beginning strictly before \(\tau\), and then
restricting to \([\tau,T]\), gives
\[
\omega\in
C^{1+\beta/2,\,2+\beta}
([\tau,T]\times\overline\Omega).
\]
Together with \eqref{eq:positive-time-phi-holder}, this proves the assertion.
\end{proof}

\section{Global well-posedness in $W^{1,q}$}

The local existence time in Proposition~\ref{prop:local} depends only on a bound for the current \(W^{1,q}\)-norm.  This yields the following standard continuation criterion.

\begin{lemma}[Continuation criterion]
\label{lem:continuation}
Let
\[
\omega\in C([0,T_{\max});W^{1,q}(\Omega))
\]
be the maximal mild solution obtained from Proposition~\ref{prop:local}.  If
\(T_{\max}<\infty\), then
\[
\limsup_{t\uparrow T_{\max}}
\norm{\omega(t)}{W^{1,q}}
=\infty.
\]
Equivalently, a uniform \(W^{1,q}\)-bound on the maximal existence interval
precludes finite-time breakdown.
\end{lemma}

\begin{proof}
Assume, to the contrary, that \(T_{\max}<\infty\) and
\[
\sup_{0\le t<T_{\max}}
\norm{\omega(t)}{W^{1,q}}
\le R_*<\infty.
\]
Choose \(t_n\uparrow T_{\max}\).  By time translation and
Proposition~\ref{prop:local}, the problem with initial datum \(\omega(t_n)\)
admits a mild solution on an interval of length
\(\tau=\tau(R_*)>0\), independently of \(n\).  Local uniqueness identifies this
restarted solution with the original one on their common interval.  For
\(n\) sufficiently large, \(t_n+\tau>T_{\max}\), so the restarted solution
extends \(\omega\) beyond \(T_{\max}\), contradicting maximality.
\end{proof}

Accordingly, global continuation reduces to a finite-time
\(W^{1,q}\) a priori estimate.  The first ingredient is a maximum-principle bound
that is independent of the Sobolev norm.

\subsection{Maximum principle}

Let \(T_{\max}\) denote the maximal existence time obtained by iteration of
Proposition~\ref{prop:local}, with the exponent fixed as in
Remark~\ref{rem:large-q}.  We derive estimates on \([0,T_{\max})\) that depend only
on the initial range of \(\omega\) and the fixed coefficients.

For subsequent use, define
\[
m_0:=\min_{\overline\Omega}\omega_0,
\qquad
M_0:=\max_{\overline\Omega}\omega_0.
\]

\begin{lemma}[Maximum principle]
\label{lem:maximum-principle}
Let
\[
\omega\in C([0,T_{\max});W^{1,q}(\Omega))
\]
be the maximal mild solution obtained in Proposition~\ref{prop:local}, with \(q\)
chosen as in Remark~\ref{rem:large-q}.  Then
\begin{equation}
\label{eq:invariant}
m_0\le\omega(t,x)\le M_0,
\qquad
0\le t<T_{\max},\quad x\in\overline\Omega.
\end{equation}
\end{lemma}

\begin{proof}
Fix \(0<\tau<T<T_{\max}\).  By
Lemma~\ref{lem:positive-time-classical}, \(\omega\) is classical on
\([\tau,T]\times\overline\Omega\).  Define
\[
M(t):=\max_{\overline\Omega}\omega(t,\cdot),
\qquad
m(t):=\min_{\overline\Omega}\omega(t,\cdot).
\]
At a spatial maximum \(x_M\) of \(\omega(t,\cdot)\), the Neumann condition and
the second-derivative test give
\[
\nabla\omega(t,x_M)=0,
\qquad
L_p\omega(t,x_M)\le0.
\]
Thus \eqref{eq:omega-reduced} implies
\[
\partial_t\omega(t,x_M)
\le -(M(t)-\bar\omega(t)).
\]
Similarly, at a spatial minimum \(x_m\),
\[
\nabla\omega(t,x_m)=0,
\qquad
L_p\omega(t,x_m)\ge0,
\]
and hence
\[
\partial_t\omega(t,x_m)
\ge \bar\omega(t)-m(t).
\]
The Dini-derivative formulas collected in Lemma~\ref{lem:Dini-max} of
Appendix~\ref{app:dini} yield
\[
D^+M(t)\le -(M(t)-\bar\omega(t)),
\qquad
D_+m(t)\ge \bar\omega(t)-m(t).
\]
Since \(p>0\),
\[
m(t)\le\bar\omega(t)\le M(t),
\]
and therefore
\[
D^+M(t)\le0,
\qquad
D_+m(t)\ge0.
\]
The Dini monotonicity criterion shows that \(M\) is nonincreasing and \(m\) is
nondecreasing on \([\tau,T]\).  Consequently,
\begin{equation}
m(\tau)\le\omega(t,x)\le M(\tau),
\qquad
\tau\le t\le T.
\label{eq:maximum-principle-tau}
\end{equation}

Because \(q>d\), the embedding
\[
W^{1,q}(\Omega)\hookrightarrow C(\overline\Omega)
\]
is continuous.  Hence
\[
\omega(t)\to\omega_0
\qquad\text{uniformly on }\overline\Omega
\quad\text{as }t\downarrow0,
\]
and thus
\[
M(\tau)\to M_0,
\qquad
m(\tau)\to m_0.
\]
Letting \(\tau\downarrow0\) in \eqref{eq:maximum-principle-tau} proves
\eqref{eq:invariant} on \([0,T]\).  Since \(T<T_{\max}\) is arbitrary, the
estimate holds throughout the maximal existence interval.
\end{proof}

Set
\[
K_0:=M_0-m_0.
\]
Then
\begin{equation}
\label{eq:meanosc}
\norm{\omega(t)-\bar\omega(t)}{L^\infty}\le K_0.
\end{equation}
Applying the elliptic estimate to
\[
L_p\phi=-(\omega-\bar\omega)
\]
gives, for every fixed \(r>d\),
\[
\norm{\phi(t)}{W^{2,r}}
\le C_r\norm{\omega(t)-\bar\omega(t)}{L^r}
\le C_rK_0.
\]
Hence Sobolev embedding yields the uniform bound
\begin{equation}
\label{eq:phi-grad-global}
\sup_{0\le t<T_{\max}}
\norm{\nabla\phi(t)}{L^\infty}
\le C_0,
\end{equation}
where \(C_0\) depends only on the initial oscillation and the fixed coefficients.

\subsection{Global \(W^{1,q}\) estimate}

A direct estimate of \eqref{eq:mild} in the variable \(\omega\) leads to a quadratic Volterra inequality of the schematic form
\[
Y(t)\lesssim C+\int_0^t(t-s)^{-\theta}Y(s)^2\,ds,
\qquad
Y(t)=\norm{\omega(t)}{W^{1,q}},
\]
which is sufficient for local theory but does not by itself rule out finite-time
growth.  It is therefore advantageous, at the level of a priori estimates, to introduce the exponential variable.

Define
\[
u:=e^{-\omega}.
\]
By \eqref{eq:invariant},
\begin{equation}
\label{eq:u-positive}
e^{-M_0}
\le u(t,x)\le
e^{-m_0}
\qquad (0\le t<T_{\max}).
\end{equation}
By Lemma~\ref{lem:positive-time-classical}, the following computation is valid
classically for every \(t>0\). Using
\[
\nabla u=-u\nabla\omega,
\qquad
L_pu=-uL_p\omega+u|\nabla\omega|^2,
\]
and the \(\omega\)-equation, the quadratic gradient term cancels and we obtain
\begin{equation}
\label{eq:u-equation}
u_t
=
L_pu
-\nabla\phi\cdot\nabla u
+(\omega-\bar\omega)u.
\end{equation}
The corresponding variation-of-constants formula holds first on
\([\tau,T]\), \(\tau>0\), and then extends to \(t=0\) by the continuity of
\(\omega\) in \(W^{1,q}\) and of \(u=e^{-\omega}\) in the same space. Thus
\begin{equation}
\label{eq:u-duhamel}
u(t)
=
S(t)u_0
+
\int_0^t
S(t-s)
\left[
-\nabla\phi(s)\cdot\nabla u(s)
+
(\omega(s)-\bar\omega(s))u(s)
\right]ds.
\end{equation}

By \eqref{eq:meanosc}, \eqref{eq:phi-grad-global}, and \eqref{eq:u-positive},
\begin{align}
&\norm{
-\nabla\phi(s)\cdot\nabla u(s)
+
(\omega(s)-\bar\omega(s))u(s)
}{L^q}
\nonumber\\
&\qquad
\le
C_0\norm{\nabla u(s)}{L^q}
+
K_0\norm{u(s)}{L^q}
\le C_1\norm{u(s)}{W^{1,q}}.
\label{eq:u-nonlinearity-Lq}
\end{align}
Using the boundedness of \(S(t)\) on \(W^{1,q}\) for the initial term and
\eqref{eq:semi-Lq} for the Duhamel term, we find, for \(0\le t\le T<T_{\max}\),
\begin{equation}
\label{eq:u-singular-gronwall}
Y_u(t)
\le A_T
+
B\int_0^t(t-s)^{-1/2}Y_u(s)\,ds,
\qquad
Y_u(t):=\norm{u(t)}{W^{1,q}},
\end{equation}
where \(A_T\) and \(B\) are finite. On each fixed finite interval, one may take
\(A_T=C\norm{u_0}{W^{1,q}}\) up to a harmless semigroup constant.

The singular Gr\"onwall inequality in Lemma~\ref{lem:singular-gronwall} of
Appendix~\ref{app:gronwall} gives
\begin{equation}
\label{eq:u-global-W1q}
\sup_{0\le t\le T}\norm{u(t)}{W^{1,q}}
\le C_T<\infty.
\end{equation}
Finally, by \eqref{eq:u-positive},
\[
\nabla\omega=-\frac{\nabla u}{u},
\]
so
\[
\norm{\nabla\omega(t)}{L^q}
\le e^{M_0}\norm{\nabla u(t)}{L^q}.
\]
Together with the \(L^\infty\) bound \eqref{eq:invariant},
\begin{equation}
\label{eq:omega-global-W1q}
\sup_{0\le t\le T}
\norm{\omega(t)}{W^{1,q}}
\le C_T<\infty
\qquad
\text{for every finite }T<T_{\max}.
\end{equation}
Lemma~\ref{lem:continuation} therefore implies
\[
T_{\max}=\infty.
\]

\begin{proposition}[Global Sobolev mild well-posedness]
\label{prop:global-mild}
Assume the hypotheses of Theorem~\ref{thm:main}, and fix \(q\) as in
Remark~\ref{rem:large-q}.  Then \eqref{eq:system} admits a unique global mild
solution
\[
\omega\in C([0,\infty);W^{1,q}(\Omega)),
\qquad
\phi\in C([0,\infty);W^{2,q}(\Omega)).
\]
For every \(T>0\),
\[
\sup_{0\le t\le T}
\left(
\norm{\omega(t)}{W^{1,q}}
+
\norm{\phi(t)}{W^{2,q}}
\right)<\infty,
\]
and \eqref{eq:invariant} holds for all \(t\ge0\).  Moreover, the solution depends
continuously on the initial datum on every finite time interval in the
\(W^{1,q}\)-topology.
\end{proposition}

\begin{proof}
Proposition~\ref{prop:local} yields a unique maximal mild solution on
\([0,T_{\max})\).  Lemma~\ref{lem:maximum-principle} provides the invariant bound
\eqref{eq:invariant}, while the transformation \(u=e^{-\omega}\) and the singular
Gr\"onwall estimate yield \eqref{eq:omega-global-W1q}.  The continuation criterion
Lemma~\ref{lem:continuation} therefore excludes \(T_{\max}<\infty\), and hence
\(T_{\max}=\infty\).

The elliptic estimate \eqref{eq:ellW2q} applied to
\(\phi(t)=\Phi[\omega(t)]\) gives the asserted \(W^{2,q}\)-bound for \(\phi\)
on every finite time interval.  Global uniqueness follows from local uniqueness by
a finite continuation argument on \([0,T]\).  Iterating the local
continuous-dependence estimate in the same manner yields continuous dependence on
the initial datum.
\end{proof}

\section{Global classical well-posedness}

It remains to upgrade the global Sobolev mild solution to a classical solution up to
the initial time.  Positive-time regularity is provided by
Lemma~\ref{lem:positive-time-classical}, while the
\(C^{2+\alpha}\) regularity and Neumann compatibility of the initial datum allow
standard parabolic Schauder theory to be applied near \(t=0\).

\begin{theorem}[Classical regularity upgrade]
\label{thm:classical-upgrade}
Under the assumptions of Theorem~\ref{thm:main}, the global mild solution from
Proposition~\ref{prop:global-mild} is the unique global classical solution of
\eqref{eq:system}.  More precisely, for every \(T>0\) and every
\(0<\beta<\alpha\),
\[
\omega\in
C^{1+\beta/2,\,2+\beta}([0,T]\times\overline\Omega),
\qquad
\phi\in
C^{\beta/2}\!\left([0,T];C^{2+\beta}(\overline\Omega)\right).
\]
\end{theorem}

\begin{proof}
Fix \(T>0\) and \(0<\beta<\alpha\).  We prove that the global mild
solution obtained in Proposition~\ref{prop:global-mild} belongs to
\[
C^{1+\beta/2,\,2+\beta}([0,T]\times\overline\Omega),
\]
and then recover the corresponding regularity of the elliptic variable
\(\phi\).

Since \(0<\beta<\alpha<1\), we may choose \(q<\infty\) sufficiently large
so that
\begin{equation}
q>2(d+2)
\qquad\text{and}\qquad
\beta<1-\frac{2(d+2)}{q}.
\label{eq:classical-q-choice}
\end{equation}
Set
\[
r:=\frac q2.
\]
Then
\[
r>d+2,
\qquad
\gamma:=1-\frac{d+2}{r}
=1-\frac{2(d+2)}q>\beta.
\]
Because
\(\omega_0\in C^{2+\alpha}(\overline\Omega)\),
we have
\[
\omega_0\in W^{2,r}(\Omega)\cap W^{1,q}(\Omega)
\]
for every finite \(q,r\).  Moreover,
\[
\partial_n\omega_0=0
\qquad\text{on }\partial\Omega.
\]

We first obtain a preliminary Hölder regularity of \(\omega\) on the whole
closed cylinder \([0,T]\times\overline\Omega\), including \(t=0\).

Recall that the reduced equation satisfied by the mild solution is
\begin{equation}
\omega_t-L_p\omega
=
\mathcal N(\omega),
\label{eq:classical-reduced}
\end{equation}
where
\[
\mathcal N(\omega)
=
-\nabla\Phi[\omega]\cdot\nabla\omega
-|\nabla\omega|^2
-(\omega-\bar\omega).
\]
By Proposition~\ref{prop:global-mild},
\[
\omega\in C([0,T];W^{1,q}(\Omega)).
\]
Lemma~\ref{lem:elliptic} therefore gives
\[
\Phi[\omega]
\in C([0,T];W^{2,q}(\Omega))
\]
and, since \(q>d\),
\[
\nabla\Phi[\omega]
\in C([0,T];L^\infty(\Omega)).
\]
Consequently,
\[
\nabla\Phi[\omega]\cdot\nabla\omega
\in C([0,T];L^q(\Omega)).
\]
Furthermore,
\[
|\nabla\omega|^2
\in C([0,T];L^{q/2}(\Omega))
=
C([0,T];L^r(\Omega)).
\]
The weighted mean is explicitly
\begin{equation}
\bar\omega(t)
=
\frac{\displaystyle\int_\Omega p(x)\omega(t,x)\,dx}
{\displaystyle\int_\Omega p(x)\,dx},
\label{eq:classical-bar-explicit}
\end{equation}
and hence
\[
\bar\omega\in C([0,T]).
\]
Therefore
\[
\omega-\bar\omega
\in C([0,T];L^q(\Omega)).
\]
Since \(\Omega\) is bounded and \(r=q/2<q\), the embedding
\(L^q(\Omega)\hookrightarrow L^r(\Omega)\) gives
\begin{equation}
\mathcal N(\omega)
\in C([0,T];L^r(\Omega))
\subset L^r(0,T;L^r(\Omega)).
\label{eq:classical-N-Lr}
\end{equation}

Let \(A_r\) denote the Neumann realization of \(L_p\) on \(L^r(\Omega)\).
Since \(\Omega\) is \(C^{2+\alpha}\) and
\(p\in C^{2+\alpha}(\overline\Omega)\) is strictly positive, \(A_r\) is a
uniformly elliptic Neumann operator and has maximal \(L^r\)-regularity.
Its domain is
\[
D(A_r)
=
\bigl\{
u\in W^{2,r}(\Omega):
\partial_nu=0\text{ on }\partial\Omega
\bigr\}.
\]
The assumptions on the initial datum imply
\[
\omega_0\in D(A_r).
\]
In particular,
\[
\omega_0
\in
(L^r(\Omega),D(A_r))_{1-1/r,r},
\]
which is the trace space required for maximal \(L^r\)-regularity.

Now regard \(\mathcal N(\omega)\) as a known function in
\(L^r(0,T;L^r(\Omega))\).  The mild solution \(\omega\) satisfies
\[
\omega(t)
=
S(t)\omega_0
+
\int_0^tS(t-s)\mathcal N(\omega(s))\,ds.
\]
Maximal \(L^r\)-regularity for the linear problem
\[
\left\{
\begin{aligned}
u_t-L_pu&=\mathcal N(\omega),
&&\text{in }(0,T)\times\Omega,\\
\partial_nu&=0,
&&\text{on }(0,T)\times\partial\Omega,\\
u(0)&=\omega_0,
&&\text{in }\Omega,
\end{aligned}
\right.
\]
therefore yields
\begin{equation}
\omega
\in
W^{1,r}(0,T;L^r(\Omega))
\cap
L^r(0,T;W^{2,r}(\Omega)).
\label{eq:classical-maximal-reg}
\end{equation}
Indeed, the maximal-regularity solution satisfies the same
variation-of-constants formula as the given mild solution, and uniqueness for
the linear problem shows that the two coincide.

Since \(r>d+2\), the parabolic Sobolev--Morrey embedding applied to
\eqref{eq:classical-maximal-reg} gives
\begin{equation}
\omega
\in
C^{(1+\gamma)/2,\,1+\gamma}
([0,T]\times\overline\Omega),
\qquad
\gamma
=
1-\frac{d+2}{r}>0.
\label{eq:classical-first-holder}
\end{equation}
In particular, because \(\beta<\gamma\),
\begin{equation}
\nabla\omega
\in
C^{\beta/2,\,\beta}
([0,T]\times\overline\Omega).
\label{eq:classical-grad-holder}
\end{equation}
We shall also need the time Hölder continuity of \(\omega\) with values in
\(C^\beta(\overline\Omega)\).  This follows from
\eqref{eq:classical-first-holder}.  For completeness, let
\(0\le s<t\le T\).  The spatial interpolation inequality gives
\[
\|\omega(t)-\omega(s)\|_{C^\beta}
\le
C
\|\omega(t)-\omega(s)\|_{C^0}^{1-\theta}
\|\omega(t)-\omega(s)\|_{C^{1+\gamma}}^\theta,
\qquad
\theta=\frac{\beta}{1+\gamma}.
\]
The \(C^{1+\gamma}\)-norm is uniformly bounded by
\eqref{eq:classical-first-holder}, while
\[
\|\omega(t)-\omega(s)\|_{C^0}
\le
C|t-s|^{(1+\gamma)/2}.
\]
Hence
\[
\|\omega(t)-\omega(s)\|_{C^\beta}
\le
C|t-s|^{(1+\gamma-\beta)/2}.
\]
Because \(\beta<\gamma<1\),
\[
\frac{1+\gamma-\beta}{2}>\frac{\beta}{2}.
\]
After enlarging the constant on the finite interval \([0,T]\), we obtain
\begin{equation}
\|\omega(t)-\omega(s)\|_{C^\beta}
\le
C_T|t-s|^{\beta/2}.
\label{eq:classical-omega-time-holder}
\end{equation}
Thus
\begin{equation}
\omega
\in
C^{\beta/2}
\bigl([0,T];C^\beta(\overline\Omega)\bigr).
\label{eq:classical-omega-Cbeta}
\end{equation}

We next improve the regularity of the elliptic variable.  For every
\(t\in[0,T]\), \(\phi(t)=\Phi[\omega(t)]\) is the normalized solution of
\begin{equation}
\left\{
\begin{aligned}
L_p\phi(t)
&=-(\omega(t)-\bar\omega(t)),
&&\text{in }\Omega,\\
\partial_n\phi(t)&=0,
&&\text{on }\partial\Omega,\\
\int_\Omega\phi(t,x)\,dx&=0.
\end{aligned}
\right.
\label{eq:classical-elliptic}
\end{equation}
The compatibility condition holds because, by the definition of
\(\bar\omega(t)\),
\[
\int_\Omega
p(x)(\omega(t,x)-\bar\omega(t))\,dx
=0.
\]
Since \(p\in C^{2+\alpha}(\overline\Omega)\),
\(\partial\Omega\) is \(C^{2+\alpha}\), and
\(\beta<\alpha\), the elliptic Schauder estimate for the Neumann problem gives
\begin{equation}
\|\phi(t)\|_{C^{2+\beta}}
\le
C
\|\omega(t)-\bar\omega(t)\|_{C^\beta}.
\label{eq:classical-phi-schauder}
\end{equation}
Here the normalization
\(\int_\Omega\phi(t)\,dx=0\) removes the constant kernel of the Neumann
operator, so that no additional \(C^0\)-term is needed on the right-hand side.

The scalar map \(w\mapsto\bar w\) is a bounded linear functional on
\(C^\beta(\overline\Omega)\).  Indeed,
\[
|\bar w|
\le
\frac{\|p\|_{L^1}}{\int_\Omega p\,dx}
\|w\|_{C^0}
\le C\|w\|_{C^\beta}.
\]
Thus, from \eqref{eq:classical-omega-time-holder},
\[
|\bar\omega(t)-\bar\omega(s)|
\le
C
\|\omega(t)-\omega(s)\|_{C^\beta}
\le
C_T|t-s|^{\beta/2}.
\]
Subtracting the elliptic equations at times \(t\) and \(s\), we obtain
\[
L_p(\phi(t)-\phi(s))
=
-\Bigl[
\omega(t)-\omega(s)
-
\bigl(\bar\omega(t)-\bar\omega(s)\bigr)
\Bigr],
\]
with
\[
\partial_n(\phi(t)-\phi(s))=0,
\qquad
\int_\Omega(\phi(t)-\phi(s))\,dx=0.
\]
Applying the same normalized elliptic Schauder estimate therefore gives
\begin{align}
\|\phi(t)-\phi(s)\|_{C^{2+\beta}}
&\le
C
\left\|
\omega(t)-\omega(s)
-
\bigl(\bar\omega(t)-\bar\omega(s)\bigr)
\right\|_{C^\beta}
\nonumber\\
&\le
C_T|t-s|^{\beta/2}.
\label{eq:classical-phi-time-holder}
\end{align}
Consequently,
\begin{equation}
\phi
\in
C^{\beta/2}
\left(
[0,T];C^{2+\beta}(\overline\Omega)
\right).
\label{eq:classical-phi-final}
\end{equation}

We can now verify that the full nonlinear right-hand side in
\eqref{eq:classical-reduced} belongs to the parabolic Hölder space required
by Schauder theory.  From
\eqref{eq:classical-grad-holder} and
\eqref{eq:classical-phi-final},
\[
\nabla\omega
\in C^{\beta/2,\beta},
\qquad
\nabla\phi
\in C^{\beta/2,\beta}
\]
on \([0,T]\times\overline\Omega\).  The parabolic Hölder space
\(C^{\beta/2,\beta}\) is an algebra, and hence
\[
\nabla\phi\cdot\nabla\omega
\in
C^{\beta/2,\beta}
([0,T]\times\overline\Omega),
\]
and
\[
|\nabla\omega|^2
\in
C^{\beta/2,\beta}
([0,T]\times\overline\Omega).
\]
Moreover,
\[
\omega-\bar\omega
\in
C^{\beta/2,\beta}
([0,T]\times\overline\Omega).
\]
It follows that
\begin{equation}
\mathcal N(\omega)
\in
C^{\beta/2,\beta}
([0,T]\times\overline\Omega).
\label{eq:classical-N-holder}
\end{equation}

We now apply the global boundary Schauder estimate to
\begin{equation}
\left\{
\begin{aligned}
\omega_t-L_p\omega
&=\mathcal N(\omega),
&&\text{in }(0,T)\times\Omega,\\
\partial_n\omega&=0,
&&\text{on }(0,T)\times\partial\Omega,\\
\omega(0,\cdot)&=\omega_0,
&&\text{in }\Omega.
\end{aligned}
\right.
\label{eq:classical-parabolic-schauder-problem}
\end{equation}
To check the hypotheses explicitly, note that
\[
L_p
=
\Delta+\nabla\log p\cdot\nabla,
\]
and
\[
\nabla\log p
\in C^{1+\alpha}(\overline\Omega)
\subset C^\beta(\overline\Omega).
\]
The boundary \(\partial\Omega\) is \(C^{2+\alpha}\), hence
\(C^{2+\beta}\), and the initial datum satisfies
\[
\omega_0
\in C^{2+\alpha}(\overline\Omega)
\subset C^{2+\beta}(\overline\Omega).
\]
The zero-order compatibility condition at the parabolic corner is precisely
\[
\partial_n\omega_0=0
\qquad\text{on }\partial\Omega,
\]
which is one of the assumptions of Theorem~\ref{thm:main}.  Finally,
\eqref{eq:classical-N-holder} supplies the required Hölder regularity of the
forcing term.  Boundary parabolic Schauder theory therefore yields
\begin{equation}
\omega
\in
C^{1+\beta/2,\,2+\beta}
([0,T]\times\overline\Omega).
\label{eq:classical-omega-final}
\end{equation}
More precisely, one has an estimate of the form
\[
\|\omega\|_{C^{1+\beta/2,\,2+\beta}
([0,T]\times\overline\Omega)}
\le
C_T
\left(
\|\omega_0\|_{C^{2+\beta}}
+
\|\mathcal N(\omega)\|_{C^{\beta/2,\beta}
([0,T]\times\overline\Omega)}
+
\|\omega\|_{C^0([0,T]\times\overline\Omega)}
\right),
\]
where the constant depends only on
\(\Omega,p,\beta\), and \(T\).

We have therefore proved
\[
\omega
\in
C^{1+\beta/2,\,2+\beta}
([0,T]\times\overline\Omega),
\qquad
\phi
\in
C^{\beta/2}
\left([0,T];C^{2+\beta}(\overline\Omega)\right).
\]
Since \(T>0\) and \(0<\beta<\alpha\) were arbitrary, these estimates hold on
every finite time interval and for every such \(\beta\).  In particular,
\(\omega_t\), \(D^2\omega\), \(\nabla\omega\), and \(\nabla\phi\) are
continuous with the indicated Hölder regularity, so both equations in
\eqref{eq:system}, as well as the Neumann boundary conditions, are satisfied
pointwise.  Hence the global mild solution is a global classical solution.

For completeness, we finally verify uniqueness in the classical class.  Let
\((\widehat\omega,\widehat\phi)\) be another global classical solution of
\eqref{eq:system} with the same initial datum.  The elliptic equation, the
Neumann boundary condition, and the normalization imply
\[
\widehat\phi(t)=\Phi[\widehat\omega(t)]
\]
for every \(t\).  On every finite interval \([0,T]\), classical regularity
implies
\[
\widehat\omega
\in C([0,T];W^{1,q}(\Omega))
\]
for every finite \(q\).  Applying the variation-of-constants formula to its
parabolic equation shows that \(\widehat\omega\) satisfies the mild
formulation \eqref{eq:mild}.  Thus both \(\widehat\omega\) and the solution
constructed in Proposition~\ref{prop:global-mild} are \(W^{1,q}\)-mild
solutions with the same initial datum.  Local uniqueness from
Proposition~\ref{prop:local} gives equality on a nontrivial initial time
interval.  Reapplying the same local uniqueness argument at the endpoint of
that interval and continuing finitely many times on any prescribed
\([0,T]\) yields
\[
\widehat\omega=\omega
\qquad\text{on }[0,T].
\]
The elliptic uniqueness then gives
\[
\widehat\phi=\phi.
\]
Since \(T>0\) is arbitrary, the classical solution is unique globally.

There is one minor point concerning the choice of \(q\) in
\eqref{eq:classical-q-choice}.  For smooth initial data,
\(\omega_0\in W^{1,q}(\Omega)\) for every finite \(q\), so
Propositions~\ref{prop:local} and~\ref{prop:global-mild} may be applied with
any sufficiently large \(q\).  Solutions obtained using two different
exponents \(q_1<q_2\) coincide: the \(W^{1,q_2}\)-solution also belongs to
\(C([0,T];W^{1,q_1})\) because \(\Omega\) is bounded, and uniqueness in the
\(W^{1,q_1}\)-class identifies the two solutions locally; continuation gives
equality on every finite interval.  Thus choosing \(q\) depending on the
prescribed exponent \(\beta\) does not change the underlying global mild
solution.

This proves the theorem and, together with
Proposition~\ref{prop:global-mild}, completes the proof of
Theorem~\ref{thm:main}.
\end{proof}

\section{Long-time behavior and exponential convergence}
\label{sec:long-time}

This section characterizes the equilibrium selected by the conserved mass and
establishes exponential convergence of the accelerated Langevin dynamics whose
feedback velocity is constructed by the parabolic--elliptic PDE system.  We prove
two complementary forms of relaxation: uniform decay of the logarithmic density
ratio and dissipation of relative entropy.  We keep the
unnormalized target \(p\) in the evolution and use
\[
Z_p=\int_\Omega p(x)\,dx,
\qquad
\pi=\frac{p}{Z_p},
\qquad
d\mu(x):=\pi(x)\,dx,
\qquad
\int_\Omega d\mu=1
\]
whenever a probability measure is required.

\begin{theorem}[Exponential convergence to equilibrium]
\label{thm:convergence}
Let \((\omega,\phi)\) be the global classical solution of
Theorem~\ref{thm:main}, and let
\begin{equation}
\label{eq:initial-q-mass}
m_0
:=
\int_\Omega p(x)e^{-\omega_0(x)}\,dx
=
\int_\Omega q_0(x)\,dx.
\end{equation}
Define
\begin{equation}
\label{eq:omega-infty}
\omega_\infty
:=
-\log\left(\frac{m_0}{Z_p}\right)
=
\log\frac{Z_p}{m_0}.
\end{equation}
Then, for every \(t\ge0\),
\begin{equation}
\label{eq:Linfty-convergence}
\norm{\omega(t)-\omega_\infty}{L^\infty}
\le
e^{-t}\operatorname{osc}_\Omega\omega_0.
\end{equation}
In particular,
\[
\omega(t,\cdot)\longrightarrow\omega_\infty
\quad\text{and}\quad
q_t=p e^{-\omega(t,\cdot)}
\longrightarrow
\frac{m_0}{Z_p}p
=
m_0\pi
\]
uniformly as \(t\to\infty\).  If \(q_0\) is normalized, then \(m_0=1\),
\(\omega_\infty=\log Z_p\), and \(q_t\to\pi=p/Z_p\).
\end{theorem}

\begin{proof}
Set
\[
M(t):=\max_{\overline\Omega}\omega(t,\cdot),
\qquad
m(t):=\min_{\overline\Omega}\omega(t,\cdot),
\qquad
\mathcal O(t):=M(t)-m(t).
\]
By the Dini-derivative estimates established in Appendix~\ref{app:dini}
(see \eqref{eq:Dini-M-app}--\eqref{eq:Dini-m-app}),
\begin{align}
D^+M(t)
&\le -(M(t)-\bar\omega(t)),
\label{eq:Dini-M}\\
D_+m(t)
&\ge \bar\omega(t)-m(t).
\label{eq:Dini-m}
\end{align}
Consequently,
\begin{align}
D^+\mathcal O(t)
&\le D^+M(t)-D_+m(t)
\nonumber\\
&\le
-(M(t)-\bar\omega(t))
-(\bar\omega(t)-m(t))
\nonumber\\
&=
-\mathcal O(t).
\label{eq:Dini-osc-main}
\end{align}
Applying the Dini--Gr\"onwall inequality, Lemma~\ref{lem:Dini-Gronwall}, gives
\begin{equation}
\label{eq:osc-decay}
\mathcal O(t)\le e^{-t}\mathcal O(0).
\end{equation}

The conservation law gives
\begin{equation}
\label{eq:exp-mass-conservation}
\frac{d}{dt}
\int_\Omega p(x)e^{-\omega(t,x)}\,dx
=0,
\end{equation}
and hence
\begin{equation}
\label{eq:exp-mass-value}
\int_\Omega p(x)e^{-\omega(t,x)}\,dx
=
m_0
=
Z_p e^{-\omega_\infty}.
\end{equation}
Since
\[
Z_p e^{-M(t)}
\le
\int_\Omega p(x)e^{-\omega(t,x)}\,dx
\le
Z_p e^{-m(t)},
\]
we obtain
\[
e^{-M(t)}
\le
\frac{m_0}{Z_p}
=
e^{-\omega_\infty}
\le
e^{-m(t)}.
\]
Therefore \(\omega_\infty\in[m(t),M(t)]\), and
\[
\norm{\omega(t)-\omega_\infty}{L^\infty}
\le
M(t)-m(t)
\le
e^{-t}\operatorname{osc}_\Omega\omega_0,
\]
which proves \eqref{eq:Linfty-convergence}.  The stated convergence of \(q_t\)
follows from \(q_t=pe^{-\omega(t)}\).
\end{proof}

The density \(q_t\) in Theorem~\ref{thm:convergence} is precisely the density of
the accelerated Langevin dynamics introduced in Section~\ref{sec:model}.  
Therefore Theorem~\ref{thm:convergence} is an exponential convergence result for
the accelerated Langevin dynamics itself: for normalized \(q_0\), the law of
\(X_t\) converges to the normalized target \(\pi=p/Z_p\), while the logarithmic
density ratio contracts at the explicit rate \(e^{-t}\).

The normalization of \(p\) affects only the additive equilibrium level of
\(\omega\). For a probability density \(q_t\), the limiting value is
\(\log Z_p\), while the limiting density is the normalized target
\(\pi=p/Z_p\).  The evolution itself is unchanged by the unknown constant \(Z_p\).

\subsection{Exact entropy dissipation and KL convergence}

For the entropy statements, \(q\) is a probability density and the reference
probability density is
\(
\pi=p/Z_p.
\)
We use the relative Fisher information
\begin{equation}
\label{eq:relative-fisher}
I(q\Vert \pi)
:=
\int_\Omega
q\left|\nabla\log\frac{q}{\pi}\right|^2dx
=
\int_\Omega
q\left|\nabla\log\frac{q}{p}\right|^2dx
=
\int_\Omega q|\nabla r|^2dx
\end{equation}
and the Jeffreys divergence
\begin{equation}
\label{eq:jeffreys}
J(q,\pi)
:=
\operatorname{KL}(q\Vert \pi)
+
\operatorname{KL}(\pi\Vert q).
\end{equation}

\begin{theorem}[Exact Jeffreys--Fisher entropy dissipation]
\label{thm:jeffreys-fisher}
Let \((\omega,\phi)\) be the global classical solution of
Theorem~\ref{thm:main}, set
\[
q_t:=p e^{-\omega(t,\cdot)},
\]
and assume that \(q_0\) is normalized:
\[
\int_\Omega q_0\,dx=1.
\]
Then \(q_t\) remains normalized for all \(t\ge0\), and
\begin{equation}
\label{eq:jeffreys-fisher-identity}
\frac{d}{dt}\operatorname{KL}(q_t\Vert \pi)
=
-I(q_t\Vert \pi)
-J(q_t,\pi).
\end{equation}
\end{theorem}

\begin{proof}
With
\[
r:=\log\frac{q_t}{\pi}=-\omega+\log Z_p,
\]
define
\[
H(t):=\operatorname{KL}(q_t\Vert \pi)
=
\int_\Omega q_t\,r\,dx.
\]
Direct calculation gives
\begin{align}
H'(t)
&=
-\int_\Omega
r\,\nabla\cdot(q_t\nabla\phi)\,dx
+
\int_\Omega
r\,\nabla\cdot(q_t\nabla r)\,dx.
\label{eq:entropy-split}
\end{align}

For the Langevin contribution, integration by parts yields
\begin{align}
\int_\Omega
r\,\nabla\cdot(q_t\nabla r)\,dx
&=
-\int_\Omega q_t|\nabla r|^2dx
\nonumber\\
&=
-I(q_t\Vert \pi).
\label{eq:entropy-langevin-part}
\end{align}
For the weighted-Poisson transport contribution,
\begin{align}
-\int_\Omega
r\,\nabla\cdot(q_t\nabla\phi)\,dx
&=
\int_\Omega q_t\nabla r\cdot\nabla\phi\,dx.
\label{eq:entropy-wlf-1}
\end{align}
Since \(q_t=\pi e^r\),
\[
q_t\nabla r=\pi\nabla(e^r).
\]
Using \(L_\pi=L_p\), integration by parts gives
\begin{align}
\int_\Omega q\nabla r\cdot\nabla\phi\,dx
&=
\int_\Omega \pi\nabla(e^r)\cdot\nabla\phi\,dx
\nonumber\\
&=
-\int_\Omega
e^r\,\nabla\cdot(\pi\nabla\phi)\,dx
\nonumber\\
&=
-\int_\Omega q_t\,\Lp\phi\,dx
\nonumber\\
&=
-\int_\Omega q_t(r-\bar r_\pi)\,dx.
\label{eq:entropy-wlf-2}
\end{align}
with $\bar r_\pi=-\bar\omega+\log Z_p$.

Because both \(q_t\) and \(\pi\) are normalized,
\[
\int_\Omega q_t(r-\bar r_\pi)\,dx
=
\int_\Omega q_t r\,dx-\bar r_\pi.
\]
Moreover,
\[
\int_\Omega q_tr\,dx
=
\operatorname{KL}(q_t\Vert \pi),
\qquad
\bar r_\pi
=
\int_\Omega \pi\log\frac{q_t}{\pi}\,dx
=
-\operatorname{KL}(\pi\Vert q_t).
\]
Hence
\begin{equation}
\label{eq:entropy-wlf-final}
-\int_\Omega q_t(r-\bar r_\pi)\,dx
=
-
\left[
\operatorname{KL}(q_t\Vert \pi)
+
\operatorname{KL}(\pi\Vert q_t)
\right]
=
-J(q_t,\pi).
\end{equation}
Combining \eqref{eq:entropy-split}, \eqref{eq:entropy-langevin-part}, and
\eqref{eq:entropy-wlf-final} proves
\eqref{eq:jeffreys-fisher-identity}.
\end{proof}

\begin{corollary}[Exponential convergence in relative entropy]
\label{cor:entropy-rates}
Under the assumptions of Theorem~\ref{thm:jeffreys-fisher},
\begin{equation}
\label{eq:unconditional-kl-rate}
\operatorname{KL}(q_t\Vert \pi)
\le
e^{-t}\operatorname{KL}(q_0\Vert \pi),
\qquad t\ge0.
\end{equation}
\end{corollary}

\begin{proof}
Since relative entropy is nonnegative,
\[
J(q,\pi)
=
\operatorname{KL}(q\Vert \pi)
+
\operatorname{KL}(\pi\Vert q)
\ge
\operatorname{KL}(q\Vert \pi).
\]
Thus \eqref{eq:jeffreys-fisher-identity} implies
\[
\frac{d}{dt}\operatorname{KL}(q_t\Vert \pi)
\le
-\operatorname{KL}(q_t\Vert \pi),
\]
and Gr\"onwall's inequality gives \eqref{eq:unconditional-kl-rate}.
\end{proof}


\section{Conclusion}
\label{sec:conclusion}

We have analyzed an accelerated Langevin dynamics in which the classical local
Langevin relaxation is coupled to a nonlocal gradient velocity reconstructed from
a target-weighted Poisson equation.  The target \(p\) may be unnormalized.  On a
bounded \(C^{2+\alpha}\) domain with homogeneous Neumann boundary conditions, the
logarithmic ratio \(\omega=\log(p/q)\) leads to a parabolic--elliptic system for
which we established local well-posedness in \(W^{1,q}\), positive-time
regularization, and global classical well-posedness.  The key a priori ingredients
are the maximum principle for \(\omega\) and the global \(W^{1,q}\) control
obtained after the transformation \(u=e^{-\omega}\).

The evolution is invariant under multiplication of \(p\) by a positive constant.
Writing
\[
Z_p=\int_\Omega p\,dx,
\qquad
\pi=\frac{p}{Z_p},
\]
the conserved quantity
\(
\int_\Omega p e^{-\omega}\,dx
=
\int_\Omega q_t\,dx
\)
selects the limiting additive constant of \(\omega\), while its spatial
oscillation contracts at the explicit rate \(e^{-t}\).  For normalized solution
densities, \(\omega(t)\to\log Z_p\) and \(q_t\to\pi\) uniformly.  At the entropy
level, the exact Jeffreys--Fisher dissipation identity is naturally written
relative to the normalized target \(\pi\): the Langevin diffusion dissipates
\(I(q_t\Vert\pi)\), whereas the weighted-Poisson transport dissipates
\(J(q_t,\pi)\).  Consequently,
\(\operatorname{KL}(q_t\Vert\pi)\) decays at least as \(e^{-t}\) without a
convexity assumption on the target.

These results provide a continuous-time analytical foundation for combining local
diffusive sampling with nonlocal weighted-Laplacian transport.  Natural directions
for further study include extensions to unbounded domains and weaker target
regularity, particle
approximation and quantitative numerical investigations of the resulting
acceleration especially for high dimensional problems.

\appendix

\section{Heat-semigroup estimates for the weighted diffusion}
\label{app:semigroup}

This appendix derives the estimates
\begin{align*}
\norm{S(t)g}{W^{1,q}}
&\lesssim t^{-1/2}\norm{g}{L^q},\\
\norm{S(t)g}{W^{1,q}}
&\lesssim
t^{-\frac12-\frac{d}{2q}}\norm{g}{L^{q/2}},
\end{align*}
used in the fixed-point argument.

\subsection{Heat-kernel bounds}

The Neumann realization of
\[
L_p=\Delta+\nabla\log p\cdot\nabla
\]
is a uniformly elliptic second-order operator with smooth bounded lower-order
coefficients on a bounded smooth domain. Its heat semigroup has a kernel
\(K(t,x,y)\) satisfying, for \(0<t\le1\),
\begin{align}
|K(t,x,y)|
&\le
C t^{-d/2}
\exp\left(-\frac{|x-y|^2}{ct}\right),
\label{eq:kernel-Gaussian}\\
|\nabla_xK(t,x,y)|
&\le
C t^{-(d+1)/2}
\exp\left(-\frac{|x-y|^2}{ct}\right).
\label{eq:kernel-gradient}
\end{align}
These are the standard Gaussian and gradient Gaussian bounds for the Neumann heat
kernel. The exponent calculations below are an elementary consequence of
\eqref{eq:kernel-Gaussian}--\eqref{eq:kernel-gradient}.

Extend \(g\) by zero outside \(\Omega\), denoting the extension by \(\widetilde g\).
Define on \(\R^d\)
\[
H_t(z):=
Ct^{-d/2}e^{-|z|^2/(ct)},
\qquad
G_t(z):=
Ct^{-(d+1)/2}e^{-|z|^2/(ct)}.
\]
Then
\begin{align}
|S(t)g(x)|
&\le (H_t*|\widetilde g|)(x),
\label{eq:conv-H}\\
|\nabla S(t)g(x)|
&\le (G_t*|\widetilde g|)(x).
\label{eq:conv-G}
\end{align}

\subsection{The \(L^a\) norms of the Gaussian kernels}

Let \(1\le a\le\infty\). For finite \(a\),
\begin{align*}
\norm{H_t}{L^a(\R^d)}^a
&=
C^a t^{-da/2}
\int_{\R^d}e^{-a|z|^2/(ct)}\,dz.
\end{align*}
Set \(z=\sqrt t\,\xi\). Then \(dz=t^{d/2}d\xi\), and therefore
\[
\norm{H_t}{L^a}^a
=
C_a t^{-da/2+d/2}.
\]
Taking the \(a\)-th root gives
\begin{equation}
\label{eq:Hta}
\norm{H_t}{L^a}
\le
C_a t^{-\frac d2+\frac{d}{2a}}
=
C_a t^{-\frac d2(1-\frac1a)}.
\end{equation}

Similarly,
\begin{align*}
\norm{G_t}{L^a}^a
&=
C^a t^{-(d+1)a/2}
\int_{\R^d}e^{-a|z|^2/(ct)}\,dz\\
&=
C_a t^{-(d+1)a/2+d/2},
\end{align*}
hence
\begin{equation}
\label{eq:Gta}
\norm{G_t}{L^a}
\le
C_a
t^{-\frac12-\frac d2(1-\frac1a)}.
\end{equation}

\subsection{General \(L^r\)-to-\(L^q\) and gradient estimates}

Let \(1<r\le q<\infty\). Choose \(a\) so that Young's convolution relation holds:
\begin{equation}
\label{eq:youngrelation}
1+\frac1q=\frac1a+\frac1r.
\end{equation}
Equivalently,
\begin{equation}
\label{eq:aidentity}
1-\frac1a=\frac1r-\frac1q.
\end{equation}
Young's inequality together with \eqref{eq:conv-H} and \eqref{eq:Hta} yields
\begin{align}
\norm{S(t)g}{L^q}
&\le
\norm{H_t}{L^a}\norm{g}{L^r}
\nonumber\\
&\le
C
t^{-\frac d2(\frac1r-\frac1q)}
\norm{g}{L^r}.
\label{eq:general-LrLq}
\end{align}
Likewise, \eqref{eq:conv-G} and \eqref{eq:Gta} give
\begin{align}
\norm{\nabla S(t)g}{L^q}
&\le
\norm{G_t}{L^a}\norm{g}{L^r}
\nonumber\\
&\le
C
t^{-\frac12-\frac d2(\frac1r-\frac1q)}
\norm{g}{L^r}.
\label{eq:general-grad}
\end{align}
Thus, because \(0<t\le1\),
\begin{equation}
\label{eq:general-W1}
\norm{S(t)g}{W^{1,q}}
\le
C
t^{-\frac12-\frac d2(\frac1r-\frac1q)}
\norm{g}{L^r}.
\end{equation}

\subsection{The case \(r=q\)}

Taking \(r=q\) in \eqref{eq:general-LrLq} gives
\[
\norm{S(t)g}{L^q}\le C\norm{g}{L^q}.
\]
Taking \(r=q\) in \eqref{eq:general-grad} gives
\[
\norm{\nabla S(t)g}{L^q}
\le Ct^{-1/2}\norm{g}{L^q}.
\]
Hence
\[
\norm{S(t)g}{W^{1,q}}
\le
C(1+t^{-1/2})\norm{g}{L^q}.
\]
For \(0<t\le1\), \(1\le t^{-1/2}\), and therefore
\begin{equation}
\label{eq:appendix-seven}
\norm{S(t)g}{W^{1,q}}
\le
Ct^{-1/2}\norm{g}{L^q}.
\end{equation}

\subsection{The case \(r=q/2\)}

Now take \(r=q/2\). Since \(q>2\), \(r>1\), and
\[
\frac1r-\frac1q
=
\frac2q-\frac1q
=
\frac1q.
\]
Thus
\[
\norm{S(t)g}{L^q}
\le
Ct^{-d/(2q)}\norm{g}{L^{q/2}},
\]
and
\[
\norm{\nabla S(t)g}{L^q}
\le
Ct^{-\frac12-\frac{d}{2q}}
\norm{g}{L^{q/2}}.
\]
For \(0<t\le1\), the gradient term has the stronger singularity, so
\begin{equation}
\label{eq:appendix-eight}
\norm{S(t)g}{W^{1,q}}
\le
Ct^{-\theta}\norm{g}{L^{q/2}},
\qquad
\theta=\frac12+\frac{d}{2q}.
\end{equation}

The Duhamel integral is finite precisely if \(\theta<1\), namely
\[
\frac12+\frac{d}{2q}<1
\quad\Longleftrightarrow\quad
q>d.
\]
Moreover,
\begin{align}
\int_0^t(t-s)^{-\theta}\,ds
&=
\frac{t^{1-\theta}}{1-\theta}
\nonumber\\
&=
C_{d,q}t^\delta,
\qquad
\delta
=
1-\theta
=
\frac{q-d}{2q}>0.
\label{eq:time-integral-theta}
\end{align}

\subsection{Boundedness on \(W^{1,q}\)}

For the initial term in the mild formulation one does not use the singular estimate
\eqref{eq:appendix-seven}, because it would blow up as \(t\downarrow0\). Instead one
uses the strong continuity of the Neumann semigroup on \(W^{1,q}\):
\[
S(t)g\to g\quad\text{in }W^{1,q}\quad(t\downarrow0),
\]
and therefore
\[
\sup_{0\le t\le1}
\norm{S(t)g}{W^{1,q}}
\le C\norm{g}{W^{1,q}}.
\]
This is a standard consequence of the analytic-semigroup realization of \(L_p\);
equivalently, it follows from the standard \(W^{1,q}\) parabolic theory for the
homogeneous Neumann problem. This boundedness is used only for \(S(t)\omega_0\) (or
\(S(t)u_0\)); the singular smoothing estimates are used for the Duhamel terms.

\section{Singular and fractional Gr\"onwall inequalities}
\label{app:gronwall}

\subsection{The special kernel \((t-s)^{-1/2}\)}

We begin with the kernel that occurs in the global continuation estimate.

\begin{lemma}[Singular Gr\"onwall inequality with square-root kernel]
\label{lem:singular-gronwall}
Let \(Y\in C([0,T];[0,\infty))\) satisfy
\begin{equation}
\label{eq:sg-assumption}
Y(t)\le
A+B\int_0^t(t-s)^{-1/2}Y(s)\,ds,
\qquad 0\le t\le T,
\end{equation}
for some \(A,B\ge0\).  Then
\begin{equation}
\label{eq:sg-conclusion}
Y(t)
\le
A(1+2B\sqrt T)e^{\pi B^2t},
\qquad 0\le t\le T.
\end{equation}
In particular, \(Y\) is bounded on every finite time interval.
\end{lemma}

\begin{proof}
For a nonnegative function \(f\), define
\[
(Kf)(t):=
\int_0^t(t-s)^{-1/2}f(s)\,ds.
\]
Then \eqref{eq:sg-assumption} can be written as
\[
Y\le A+BKY.
\]
By Tonelli's theorem,
\begin{align*}
K^2f(t)
&=
\int_0^t f(r)
\left[
\int_r^t
(t-s)^{-1/2}(s-r)^{-1/2}\,ds
\right]dr.
\end{align*}
The change of variables \(s=r+(t-r)\xi\) gives
\[
\int_r^t
(t-s)^{-1/2}(s-r)^{-1/2}\,ds
=
\int_0^1
(1-\xi)^{-1/2}\xi^{-1/2}\,d\xi
=
\pi.
\]
Therefore
\begin{equation}
\label{eq:K2}
K^2f(t)=\pi\int_0^t f(r)\,dr.
\end{equation}
Applying \(K\) to \(Y\le A+BKY\) and using
\[
(K1)(t)=2\sqrt t
\]
together with \eqref{eq:K2}, we obtain
\[
KY(t)
\le
2A\sqrt t
+
\pi B\int_0^tY(s)\,ds.
\]
Substitution into \eqref{eq:sg-assumption} gives
\[
Y(t)
\le
A(1+2B\sqrt T)
+
\pi B^2\int_0^tY(s)\,ds.
\]
The classical Gr\"onwall inequality now yields \eqref{eq:sg-conclusion}.
\end{proof}

\subsection{General fractional Gr\"onwall inequality}

The preceding square-root kernel admits an elementary proof because its second self-convolution is constant.  The corresponding estimate for a general fractional kernel is expressed in terms of the Mittag--Leffler function.

\begin{theorem}[Fractional Gr\"onwall inequality]
\label{thm:fractional-gronwall}
Let \(0<\alpha<1\), let \(Y\in C([0,T];[0,\infty))\), and assume
\begin{equation}
\label{eq:fg-assumption}
Y(t)
\le
A+B\int_0^t(t-s)^{-\alpha}Y(s)\,ds,
\qquad 0\le t\le T,
\end{equation}
for some \(A,B\ge0\).  Set \(\beta:=1-\alpha\in(0,1)\).  Then
\begin{equation}
\label{eq:fg-result}
Y(t)
\le
A E_\beta\!\left(
B\Gamma(\beta)t^\beta
\right),
\qquad 0\le t\le T,
\end{equation}
where
\[
E_\beta(z)
:=
\sum_{k=0}^{\infty}
\frac{z^k}{\Gamma(\beta k+1)}
\]
is the Mittag--Leffler function.  In particular, \(Y\) is bounded on every finite
time interval.
\end{theorem}

\begin{proof}
Define
\[
g_\beta(t):=
\frac{t^{\beta-1}}{\Gamma(\beta)}.
\]
Since \(\beta-1=-\alpha\), inequality \eqref{eq:fg-assumption} is equivalent to
\[
Y\le A+\lambda\,g_\beta*Y,
\qquad
\lambda:=B\Gamma(\beta).
\]
For \(\beta,\gamma>0\), the beta-function identity gives
\begin{equation}
\label{eq:gbeta-conv}
g_\beta*g_\gamma=g_{\beta+\gamma}.
\end{equation}
Indeed,
\begin{align*}
(g_\beta*g_\gamma)(t)
&=
\frac1{\Gamma(\beta)\Gamma(\gamma)}
\int_0^t
(t-s)^{\beta-1}s^{\gamma-1}\,ds\\
&=
\frac{t^{\beta+\gamma-1}}
{\Gamma(\beta)\Gamma(\gamma)}
B(\beta,\gamma)\\
&=
\frac{t^{\beta+\gamma-1}}
{\Gamma(\beta+\gamma)}.
\end{align*}
Consequently, \(g_\beta^{*k}=g_{k\beta}\).

Iterating the inequality \(N\) times yields
\begin{equation}
\label{eq:iteration-fg}
Y(t)
\le
A\sum_{k=0}^{N-1}
\lambda^k(g_{k\beta}*1)(t)
+
\lambda^N(g_{N\beta}*Y)(t),
\end{equation}
where the \(k=0\) term is interpreted as \(1\).  For \(k\ge1\),
\[
(g_{k\beta}*1)(t)
=
\frac{t^{k\beta}}{\Gamma(k\beta+1)}.
\]
Since \(Y\) is continuous on \([0,T]\), it is bounded there, say by \(M_T\).
Hence
\[
0\le
\lambda^N(g_{N\beta}*Y)(t)
\le
M_T
\frac{\lambda^N t^{N\beta}}{\Gamma(N\beta+1)}
\longrightarrow0
\]
as \(N\to\infty\), uniformly for \(t\in[0,T]\).  Passing to the limit in
\eqref{eq:iteration-fg} gives
\[
Y(t)
\le
A\sum_{k=0}^{\infty}
\frac{\lambda^k t^{k\beta}}
{\Gamma(k\beta+1)}
=
A E_\beta(\lambda t^\beta),
\]
which is \eqref{eq:fg-result}.
\end{proof}

\subsection{Connection with the square-root kernel}

For \(\alpha=1/2\), one has \(\beta=1/2\) and
\[
\Gamma(1/2)=\sqrt\pi.
\]
Thus Theorem~\ref{thm:fractional-gronwall} gives
\[
Y(t)
\le
A E_{1/2}(B\sqrt{\pi t}).
\]
Lemma~\ref{lem:singular-gronwall} gives the more elementary, slightly coarser bound
\[
Y(t)
\le
A(1+2B\sqrt T)e^{\pi B^2t}.
\]
Both are sufficient for the continuation argument.

\section{Dini-derivative formulation of the oscillation estimate}
\label{app:dini}

This appendix establishes the differential inequality
\[
D^+\operatorname{osc}_\Omega\omega(t)
\le
-\operatorname{osc}_\Omega\omega(t)
\]
used in Section~\ref{sec:long-time}. We first record the Dini-derivative formulas for time-dependent extrema, then evaluate the equation at spatial extrema, and finally apply a Dini version of Gr\"onwall's inequality.

\subsection{Upper and lower right Dini derivatives}

For a real-valued function \(g\) on an interval, define
\begin{align}
D^+g(t)
&:=
\limsup_{h\downarrow0}
\frac{g(t+h)-g(t)}{h},
\label{eq:Dini-upper-def}\\
D_+g(t)
&:=
\liminf_{h\downarrow0}
\frac{g(t+h)-g(t)}{h}.
\label{eq:Dini-lower-def}
\end{align}
If \(g\) is differentiable, both quantities equal \(g'(t)\).

\begin{lemma}[Dini derivatives of spatial extrema]
\label{lem:Dini-max}
Let \(v\in C^1([0,T]\times\overline\Omega)\), and define
\[
M_v(t):=\max_{x\in\overline\Omega}v(t,x),
\qquad
m_v(t):=\min_{x\in\overline\Omega}v(t,x).
\]
Then
\begin{equation}
\label{eq:Dini-max-general}
D^+M_v(t)
\le
\max_{x\in\operatorname{Argmax}v(t,\cdot)}
\partial_t v(t,x),
\end{equation}
and
\begin{equation}
\label{eq:Dini-min-general}
D_+m_v(t)
\ge
\min_{x\in\operatorname{Argmin}v(t,\cdot)}
\partial_t v(t,x).
\end{equation}
\end{lemma}

\begin{proof}
We prove \eqref{eq:Dini-max-general}; the argument for the minimum is analogous.
Choose \(h_n\downarrow0\) such that
\[
\frac{M_v(t+h_n)-M_v(t)}{h_n}
\longrightarrow D^+M_v(t),
\]
and choose \(x_n\in\overline\Omega\) with
\(M_v(t+h_n)=v(t+h_n,x_n)\).  Since \(v(t,x_n)\le M_v(t)\),
\begin{equation}
\frac{M_v(t+h_n)-M_v(t)}{h_n}
\le
\frac{v(t+h_n,x_n)-v(t,x_n)}{h_n}.
\label{eq:max-sequence-bound}
\end{equation}
By compactness, after passing to a subsequence,
\(x_n\to x_*\in\overline\Omega\).  Continuity of \(v\) implies
\(v(t,x_*)=M_v(t)\), so
\(x_*\in\operatorname{Argmax}v(t,\cdot)\).  The continuity of
\(\partial_t v\) then gives
\[
D^+M_v(t)
\le
\partial_t v(t,x_*)
\le
\max_{x\in\operatorname{Argmax}v(t,\cdot)}
\partial_t v(t,x).
\]

For the minimum, choose \(h_n\downarrow0\) so that the corresponding lower
difference quotients converge to \(D_+m_v(t)\), and choose
\(y_n\in\operatorname{Argmin}v(t+h_n,\cdot)\).  Since
\(m_v(t)\le v(t,y_n)\),
\[
\frac{m_v(t+h_n)-m_v(t)}{h_n}
\ge
\frac{v(t+h_n,y_n)-v(t,y_n)}{h_n}.
\]
After passing to a subsequence, \(y_n\to y_*\in
\operatorname{Argmin}v(t,\cdot)\), and therefore
\[
D_+m_v(t)
\ge
\partial_t v(t,y_*)
\ge
\min_{x\in\operatorname{Argmin}v(t,\cdot)}
\partial_t v(t,x).
\]
This proves \eqref{eq:Dini-min-general}.
\end{proof}

\subsection{Evaluation of the PDE at spatial extrema}

Assume now that \(\omega\) is a classical solution. At any
\(x_M\in\operatorname{Argmax}\omega(t,\cdot)\),
\[
\nabla\omega(t,x_M)=0,
\qquad
L_p\omega(t,x_M)\le0.
\]
For an interior maximum this is immediate. If \(x_M\in\partial\Omega\), the tangential
first derivatives vanish because \(x_M\) is a boundary maximum, while the normal
derivative vanishes by the Neumann condition. Hence the full gradient is zero.
The tangential second derivatives are nonpositive, and restriction to the inward
normal curve, together with the zero normal derivative, gives a nonpositive second
normal derivative. Thus \(\Delta\omega(t,x_M)\le0\); since
\(\nabla\omega(t,x_M)=0\), also
\(L_p\omega(t,x_M)\le0\).

Evaluating
\[
\omega_t
=
L_p\omega
-\nabla\phi\cdot\nabla\omega
-|\nabla\omega|^2
-(\omega-\bar\omega)
\]
at \(x_M\) therefore gives
\[
\omega_t(t,x_M)
\le
-(M(t)-\bar\omega(t)).
\]
Lemma~\ref{lem:Dini-max} yields
\begin{equation}
\label{eq:Dini-M-app}
D^+M(t)
\le
-(M(t)-\bar\omega(t)).
\end{equation}

Likewise, at any spatial minimum \(x_m\),
\[
\nabla\omega(t,x_m)=0,
\qquad
L_p\omega(t,x_m)\ge0,
\]
and hence
\[
\omega_t(t,x_m)
\ge
-(m(t)-\bar\omega(t))
=
\bar\omega(t)-m(t).
\]
Therefore
\begin{equation}
\label{eq:Dini-m-app}
D_+m(t)
\ge
\bar\omega(t)-m(t).
\end{equation}

\subsection{Dini derivative of the oscillation}

Set
\[
\mathcal O(t):=M(t)-m(t).
\]
For arbitrary real sequences \(a_h,b_h\),
\[
\limsup_{h\downarrow0}(a_h-b_h)
\le
\limsup_{h\downarrow0}a_h
-
\liminf_{h\downarrow0}b_h.
\]
Applying this with the difference quotients of \(M\) and \(m\) gives
\begin{equation}
\label{eq:Dini-difference-rule}
D^+\mathcal O(t)
\le
D^+M(t)-D_+m(t).
\end{equation}
Combining \eqref{eq:Dini-M-app}, \eqref{eq:Dini-m-app}, and
\eqref{eq:Dini-difference-rule},
\begin{align}
D^+\mathcal O(t)
&\le
-(M-\bar\omega)-(\bar\omega-m)
\nonumber\\
&=
-(M-m)
=
-\mathcal O(t).
\label{eq:Dini-osc-app}
\end{align}
This is the desired differential inequality.

\subsection{Dini--Gr\"onwall lemma}

\begin{lemma}[Dini--Gr\"onwall inequality]
\label{lem:Dini-Gronwall}
Let \(y\in C([0,T];[0,\infty))\) satisfy
\[
D^+y(t)\le-\lambda y(t),
\qquad
0\le t<T,
\]
for some \(\lambda>0\).  Then
\[
y(t)\le e^{-\lambda t}y(0),
\qquad
0\le t\le T.
\]
\end{lemma}

\begin{proof}
Define
\[
g(t):=e^{\lambda t}y(t).
\]
The product rule for upper right Dini derivatives gives
\[
D^+g(t)
\le
e^{\lambda t}\bigl(D^+y(t)+\lambda y(t)\bigr)
\le0.
\]
We verify the corresponding monotonicity criterion.  Suppose that
\(g(b)>g(a)\) for some \(0\le a<b\le T\), and set
\[
\sigma:=\frac{g(b)-g(a)}{2(b-a)}>0,
\qquad
h(t):=g(t)-\sigma t.
\]
Then \(h(b)>h(a)\).  Let \(c\in[a,b]\) be a point at which \(h\) attains its
minimum.  Since \(h(b)>h(a)\), one has \(c<b\), and therefore
\(D^+h(c)\ge0\).  On the other hand,
\[
D^+h(c)=D^+g(c)-\sigma\le-\sigma<0,
\]
a contradiction.  Hence \(g\) is nonincreasing, and
\[
e^{\lambda t}y(t)=g(t)\le g(0)=y(0).
\]
\end{proof}

Applying Lemma~\ref{lem:Dini-Gronwall} to
\(y=\mathcal O\) and \(\lambda=1\) in \eqref{eq:Dini-osc-app} yields
\[
\operatorname{osc}_\Omega\omega(t)
\le
e^{-t}\operatorname{osc}_\Omega\omega_0.
\]

\bibliographystyle{siamplain}
\bibliography{sobolev_global_convergence_references}

\end{document}